\documentclass[hidelinks]{article}
\usepackage{graphicx} 
\usepackage{hyperref}
\usepackage{amsthm}
\usepackage{thm-restate}
\usepackage{geometry}
\usepackage{amssymb}
\usepackage{amsmath}
\usepackage{cleveref}
\usepackage{comment}
\usepackage{xcolor}
\usepackage[shortlabels]{enumitem}
\usepackage{listings}
\usepackage{blkarray}
\usepackage{lmodern}
\usepackage[english]{babel}
\usepackage{float}
\usepackage{tikz}

\newtheorem{theorem}{Theorem}[section]
\newtheorem{lemma}[theorem]{Lemma}

\newtheorem{claim}[theorem]{Claim}

\theoremstyle{definition}
\newtheorem{definition}[theorem]{Definition}

\title{Minors in small-set expanders}

\author{
Michael Krivelevich\thanks{
School of Mathematical Sciences,
Tel Aviv University, Tel Aviv 6997801, Israel. \\
Email: {\tt krivelev@tauex.tau.ac.il}. Research supported in part by NSF-BSF grant 2023688.}
\and
Rajko Nenadov\thanks{School of Computer Science, University of Auckland, New Zealand. Email: \texttt{rajko.nenadov@auckland.ac.nz}. Research supported by the Marsden Fund of the Royal Society of New Zealand.}
}
\date{}

\begin{document}

\maketitle

\begin{abstract}
We study large minors in small-set expanders. More precisely, we consider graphs with $n$ vertices and the property that every set of size at most $\alpha n / t$ expands by a factor of $t$, for some (constant) $\alpha > 0$ and large $t = t(n)$. We obtain the following:

\begin{itemize}
\item Improving results of Krivelevich and Sudakov, we show that a small-set expander contains a complete minor of order $\sqrt{n t / \log n}$.

\item We show that a small-set expander contains every graph $H$ with $O(n \log t / \log n)$ edges and vertices as a minor. We complement this with an upper bound showing that if an $n$-vertex graph $G$ has average degree $d$, then there exists a graph with $O(n \log d / \log n)$ edges and vertices which is not a minor of $G$. This has two consequences: (i) It implies the optimality of our result in the case $t = d^c$ for some constant $c > 0$, and (ii) it shows expanders are optimal minor-universal graphs of a given average degree.
\end{itemize}
\end{abstract}

\section{Introduction}

A graph $H$ is a \emph{minor} of a graph $G$ if one can obtain $H$ from
$G$ by a series of contractions of edges and deletions of vertices and edges. Equivalently, and this is the point of view we will take, $H$ is a minor of $G$ if there are $v(H)$ disjoint subsets $\{W_h \subseteq V(G)\}_{h \in V(H)}$ such that each $W_h$ induces a
connected subgraph of $G$ and there is an edge between $W_h$ and $W_{h'}$ for every $\{h, h'\} \in E(H)$. The notion of minors is central in structural graph theory. For example, the celebrated Wagner's theorem \cite{wagner37planar} states that a graph is planar if an only if it does not contain $K_5$ or $K_{3,3}$ as minors.

We study existence of large graphs as minors in vertex expanders. This line of research has been initiated by Krivelevich and Sudakov \cite{krivelevich09minors}. A particular notion of the expansion they used is that of \emph{$(\alpha, t)$-expanders}. 

\begin{definition}[$(\alpha, t)$-expander]
We say a graph $G$ with $n$ vertices is an $(\alpha, t)$-expander, for some $t \ge 1$ and $0 < \alpha < 1$, if for every $X \subseteq V(G)$ of size $|X| \le \alpha n / t$ we have 
$$
    |N(X)| \ge t|X|,
$$
where $N(X)$ denotes the external neighbourhood:
$$
    N(X) = \{v \in V(G) \setminus X \colon \{v, x\} \in E(G) \text{ for some } x \in X\}.
$$
\end{definition}
Observe that the definition of $(\alpha, t)$-expanders is only meaningful when $\alpha < 1$. The particular regime of parameters that we are interested in is when $t$ is large, potentially being even linear in the number of vertices, and $\alpha$ is constant. There is abundance of $(\alpha, t)$-expanders:
\begin{itemize}
\item We say that a graph $G$ is an \emph{$(n, d, \lambda)$-graph} if it is an $n$-vertex $d$-regular graph with the second largest absolute eigenvalue of its adjacency matrix at most $\lambda$. It  follows  from the Expander Mixing Lemma \cite{alon88mixing} that if $G$ is an $(n,d,\lambda)$-graph with, say, $\lambda < d/4$, then it is a $(\Theta(1), \Theta(d^2/\lambda^2))$-expander (see Lemma \ref{lemma:ndl_expander}). Consequently, a \emph{random $d$-regular graph} with $n$ vertices, where $d_0 \le d < n/2$ for a sufficiently large constant $d_0$, is with high probability a $(\Theta(1), \Theta(d))$-expander (see \cite{broder98randomreg,sarid23gap,tikhomirov19randomreg}). 

\item A \emph{$d$-out random graph} with $n$ vertices, a graph obtained by taking for each vertex $d$ randomly and independently chosen edges incident to it, is with high probability a $(\Theta(1), \Theta(d))$-expander for $4 \le d < n/2$ (see, e.g., Lemma 17.17 in the online version of \cite{frieze16introduction}).

\item It is well known that the binomial random graph $G(n,p)$, with $p \ge C \log n / n$, for a constant $C > 1$, is typically an $(\Theta(1), \Theta(np))$-expander. For $p < \log n / n$ the random graph $G(n,p)$ is typically not an expander due to some vertices being isolated. However, in this case one can remove a few vertices such that the resulting graph is again a $(\Theta(1), \Theta(np))$-expander (see Lemma \ref{lemma:one2all}).
\end{itemize}
To summarise, the notion of an $(\alpha, t)$-expander abstracts out one particular, very important aspect of a large class of graphs. As such it provides a framework for studying them in a unified way.

\paragraph{Complete minors.} A topic that has attracted significant attention is determining the size of a largest complete minor as a function of some other graph parameter. The most important open problem in this direction is Hadwiger's conjecture \cite{hadwiger43conj}. This conjecture states that $G$ contains $K_{\chi(G)}$ as a minor, where $\chi(G)$ denotes the chromatic number of $G$. Bollob\'as, Caitlin, and Erd\H{o}s \cite{bollobas80clique} showed that Hadwiger's conjecture holds for almost all graphs. That is, with high probability, the largest complete minor in the uniformly chosen random graph $G(n, 1/2)$ is of order $\Theta( n / \sqrt{\log n})$, whereas its chromatic number is of order $\Theta(n / \log n)$ (e.g.\ see \cite{bollobas88chromatic}). Kostochka \cite{kostochka84hadwiger} and Thomason \cite{thomason84hadwiger} obtained a bound on the largest complete minor of order $\Omega(\chi(G) / \sqrt{\chi(G)})$. In the recent breakthroughs by Norin, Postle, and Song \cite{norin23hadwiger} and Delcourt and Postle \cite{delcourt2024reducing}, this has been improved to $\Omega(\chi(G) / \log \log (\chi(G)))$. We remark that the results of Kostochka \cite{kostochka84hadwiger} and Thomason \cite{thomason84hadwiger} actually show a tight lower bound on the order of magnitude of the largest complete minor as a function of the average degree, which then implies the bound in terms of the chromatic number. For a recent short proof of this result, see \cite{alon23shortminor}.

A related problem is the study of complete minors in graphs with some given (structural) properties. This includes graphs with large girth \cite{diestel05dense,krivelevich09minors,kuhn03minors}, $H$-free graphs for various fixed graph $H$ \cite{bucic22hfree,dvovrak2019independence,krivelevich09minors,kuhn05kss}, graphs without small vertex separators \cite{alon90separator,plotkin94shallow}, graphs without a large independent set \cite{balogh11independence,fox10independence}, lifts of graphs \cite{drier06lifts}, sparse random graphs \cite{fountoulakis08sparse} and random regular graphs \cite{fountoulakis09regular}, random subgraphs of graphs with large degree \cite{erde21randomsubgraphs}. This list, which is by no means exhaustive, illustrates the importance of the topic. Some of these results were unified and extended by the authors \cite{krivelevich21sparsecut}, who studied large complete minors in graphs with strong edge-expansion properties. Here we focus on graphs with strong vertex-expansion properties. Edge-expansion and vertex-expansion are, in general, incomparable.

Large complete minors in expanders were first studied by Krivelevich and Sudakov \cite{krivelevich09minors}. Two of the main results \cite[Corollary 4.2 \& Proposition 4.3]{krivelevich09minors}, combined, read as follows.

\begin{theorem}[\cite{krivelevich09minors}] \label{thm:benny_michael}
    Let $G$ be an $(\alpha, t)$-expander with $n$ vertices, for some $t \ge 10$ and $0 < \alpha < 1$. Then the largest complete minor in $G$ is of order
    $$
        \Omega_\alpha\left( \frac{\sqrt{nt \log t}}{\log n}  + \sqrt{\frac{n \log t}{\log n}}   \right)
    $$
\end{theorem}

For $t \ge \log n$, the bound in the previous theorem evaluates to $\frac{\sqrt{nt \log t}}{\log n}$, whereas for $t < \log n$ it gives $\sqrt{n \log t / \log n}$. Our first main result improves upon this in the regime where $t$ is superconstant or subpolynomial.

\begin{theorem} \label{thm:complete_minors}
    For every $0 < \alpha < 1$ there exist $\xi > 0$ and $n_0, t_0 \in \mathbb{N}$ such that the following holds. Suppose $G$ is an $(\alpha, t)$-expander with $n \ge n_0$ vertices, for some $t_0 \le t \le \sqrt{n}$. Then $G$ contains a complete minor of order
    $$
        \xi \sqrt{\frac{nt}{\log n}}.
    $$
\end{theorem}

The upper bound $t < \sqrt{n}$ is somewhat arbitrary and can be significantly weakened. However, as that regime is already taken care of by Theorem \ref{thm:benny_michael} we did not try to optimise it.  As remarked earlier, we treat $\alpha$ as a constant and take the expansion factor $t$ and the number of vertices $n$ to be sufficiently large with respect to it. Inspection of the proof shows that this dependency is polynomial. Since it is not of our primary concern, we have decided to avoid further technical difficulties by not computing it precisely.

The authors \cite{krivelevich21sparsecut} showed that if $G$ has average degree $d$ and strong edge-expansion properties, then it contains a complete minor of order $\sqrt{n d / \log d}$. Random $d$-regular graphs, for $3 \le d \le n/2$, show that this is tight \cite{fountoulakis09regular,krivelevich21sparsecut}. It remains an interesting question whether the bound in Theorem \ref{thm:complete_minors} can be improved to $\sqrt{nt / \log t}$. This would be a rather significant improvement over the result from \cite{krivelevich21sparsecut} as in Theorem \ref{thm:complete_minors} we do not assume anything about the average degree of $G$.

\paragraph{General minors.} Instead of asking whether a graph contains a complete graph of certain size as a minor, one can also ask whether it contains some given graph $H$ as a minor. A question that has been extensively studied is the smallest average degree of a graph $G$ which enforces a particular fixed graph $H$ as a minor (by `fixed' we mean that the number of vertices of $G$ is significantly larger than the number of vertices of $H$). See \cite{chudnovsky11density,kostochka08kstminor,kostochka10dense,kuhn05forcing,myers05extremal,reed16forcing} and references therein; for recent progress, see \cite{haslegrave22extremal,hendrey22extremal}. 

Here we are interested in the largest $m$ for which a given graph is \emph{$m$-minor-universal}, that is, it contains every graph with at most $m$ vertices and at most $m$ edges as a minor. Benjamini, Kalifa, and Tzalik \cite{benjamini2025cube} studied minor-universality of hypercubes, and the optimality of their result was established by Hogan, Michel, Scott, Tamitegama, Tan, and Tsarev \cite{hogan2025tight}. Minor-universality in expanders was first studied by Kleinberg and Rubinfeld \cite{kleinberg96shortpath}, and more recently by the authors \cite{krivelevich19howtofind} and Chuzhoy and Nimavat \cite{chuzhoy2019large}. Briefly, \cite[Corollary 8.3]{krivelevich19howtofind} states that if a graph has the property that  every set of vertices of size at most $n/2$ expands by a factor of $\beta > 0$, then it is $\Omega_\beta(n / \log n)$-minor-universal, and this is optimal. This settles the case of graphs with weak vertex expansion properties. Our second main result addresses the case where $G$ has strong vertex expansion properties.

\begin{theorem} \label{thm:expanding_minors}
    For every $0 < \alpha < 1$ there exist $\xi > 0$ and $n_0, t_0 \in \mathbb{N}$ such that the following holds. Suppose $G$ is an $(\alpha, t)$-expander with $n \ge n_0$ vertices, for some $t \ge t_0$. Then $G$ is $m$-minor-universal for
    $$
        m = \xi n\; \frac{\log t}{\log n}.
    $$
\end{theorem}

We complement this with an upper bound on minor-universality for a general graph with a given average degree. A significantly weaker upper bound of order $O(nd / \log n)$ was obtained by Chuzhoy and Nimavat \cite{chuzhoy2019large}.

\begin{theorem} \label{thm:upper_bound}
Let $G$ be graph on $n \ge n_0$ vertices and of average degree $d>1$, where $n_0$ is a sufficiently large absolute constant. Then $G$ is not $m$-minor-universal for 
$$
    m = 6n\; \frac{\ln (d+2)}{\ln n}.
$$
\end{theorem}

The previous result shows not only that Theorem \ref{thm:expanding_minors} is optimal for $(\alpha, d^c)$-expanders of average degree $d$, for some constant $c > 0$, but also that these graphs are optimal graphs with the average degree $d$ from the point of view of minor-universality.

\paragraph{Structure.} The paper is organised as follows. In the next section we state  properties of small-set expanders. The main new ingredient is Lemma \ref{lemma:robust_partition} which allows us to pass to a subgraph of a small-set expander which is, again, a small-set expander and in addition there is an edge between every two large sets of vertices. In order to demonstrate ideas, in Section \ref{sec:minors} we first prove Theorem \ref{thm:expanding_minors}. The proof of Theorem \ref{thm:complete_minors} follows the same strategy, with Lemma \ref{lemma:efficient_cover} being the main difference. We finally prove Theorem \ref{thm:upper_bound} in Section \ref{sec:upper}. Throughout the paper we tacitly avoid the use of floor and ceiling. All inequalities hold with a sufficiently large margin to accommodate for this.

\section{Properties of small-set expanders}

In this section we collect some properties of small-set expanders. The key new result is Lemma \ref{lemma:robust_partition}.

Given a graph $G$ and two vertices $x,y \in V(G)$, we denote by $\textrm{dist}_G(x,y)$ the length of a shortest path from $x$ to $y$ (counting the number of edges). If $x = y$, then $\textrm{dist}_G(x, y) = 0$. More generally, for $X, Y \subseteq V(G)$, we set $\textrm{dist}_G(X, Y) = \min_{x \in X, y \in Y} \textrm{dist}_G(x,y)$.

\begin{lemma} \label{lemma:ball}
    Let $G$ be an $(\alpha, t)$-expander with $n$ vertices, for any $t \ge 1$ and $0 < \alpha < 1$. Then the ball of size $z \in \mathbb{N}_0$ around a given vertex set $U$, defined as
    $$
        B_G(U, z) = \{ v \in V(G) \colon \mathrm{dist}(U, v) \le z \},
    $$
    is of size at least $\min\{\alpha n, |U|(1 + t)^z\}$.
\end{lemma}
\begin{proof}
    We prove the claim by induction on $z$. For $z = 0$, $B_G(U, 0) = U$ thus the claim holds. Suppose it holds for some $z \ge 0$. Observe that $B_G(U, z+1) = B_G(U,z) \cup N_G(B_G(U,z))$. If $|B_G(U, z)| \ge \alpha n$, the claim trivially holds. If $\alpha n > |B_G(U, z)| > \alpha n / t$, then again $|B_G(U, z+1)| \ge \alpha n$ (this can be seen by considering an expansion of an arbitrary subset $X \subseteq B_G(U, z)$ of size $\alpha n / t$). Finally, if $|B_G(U, z)| \le \alpha n / t$, then by the induction assumption and the expansion property we have
    $$
        |B_G(U, z+1)| \ge |U|(1+t)^z + |U|(1+t)^z t = |U|(1 + t)^{z+1}. 
    $$
\end{proof}

The following result is a simple corollary of Lemma \ref{lemma:ball} (e.g.~see \cite[Lemma 5.1]{krivelevich09minors}).

\begin{lemma} \label{lemma:diameter}
    Let $G$ be a connected $(\alpha,t)$-expander with $n$ vertices, for some $t \ge 2$ and $0 < \alpha < 1$. Then the diameter of $G$ is at most $3 \log n / (\alpha \log t)$.
\end{lemma}

Another folklore fact is that if sets of size from a certain interval expand, then one can remove a small number of vertices such that the remaining graph is a small-set expander.

\begin{lemma} \label{lemma:one2all}
    Let $G$ be a graph with $n$ vertices, and suppose $|N(S)| \ge t|S|$ for every $S \subseteq V(G)$ of size $\alpha n / (2t) \le |S| \le \alpha n / t$, for some $0 < \alpha < 1$ and $t \ge 1$. Then there exists $X \subseteq V(G)$ of size $|X| \ge n - \alpha n / t$, such that $G[X]$ is an $(\alpha/4, t/2)$-expander.
\end{lemma}
\begin{proof}
    Set $X = V(G)$ and $R = \emptyset$, and repeat the following: 
    \begin{itemize}
        \item If $|R| > \alpha n / (2t)$ or $G[X]$ is an $(\alpha/4, t/2)$-expander, then terminate.
        \item Otherwise, there exists $S \subseteq X$ of size $1 \le |S| \le \alpha|X|/(2t) \le \alpha n / (2t)$ such that $|N_G(S) \cap X| < t|S|/2$. Set $X := X \setminus S$ and $R := R \cup S$. Proceed to the next iteration.
    \end{itemize}
    Since the size of $X$ decreases in each round, the procedure eventually terminates. Throughout the procedure we have $V(G) = X \cup R$.
    
    Suppose that $G[X]$ is not an $(\alpha/4, t/2)$-expander, for the final set $X$. Then $|R| > \alpha n / (2t)$. As the size of $R$ increases by at most $\alpha n / (2t)$ in each iteration, we conclude $\alpha n / (2t) < |R| \le \alpha n / t$. By the assumption of the lemma, we have $|N_G(R)| \ge t|R|$. On the other hand, by the definition of the procedure we have $|N_G(R) \cap X| = |N_G(R)| \le t|R|/2 $, which is a contradiction.
\end{proof}

The final result is a key new `pre-processing' lemma. Informally, it states that one can pass to a subgraph of a small-set expander which is again a small-set expander and additionally enjoys the property that between every two large subsets of vertices there is an edge. This property will be important in the course of proving Theorem \ref{thm:complete_minors} and Theorem \ref{thm:expanding_minors}, when arguing that certain subgraphs are connected and have small diameter.

\begin{lemma} \label{lemma:robust_partition}
    Let $G$ be an $(\alpha, t)$-expander with $n$ vertices, where $0 < \alpha < 1$ and $t > 2^{10} / \alpha$. There exists $X \subseteq V(G)$ of size $|X| \ge \alpha n / 8$ such that the following holds for $\beta = \alpha n / (8|X|)$:
    \begin{enumerate}[(1)]
        \item $G[X]$ is a $(\beta, t/4)$-expander, and
        
        \item \label{prop:edge_across} for every partition $X = R \cup A \cup B$ with $|R| \le \alpha |X|/16$ and $|A|, |B| \ge \beta |X| / 4$, there is an edge between $A$ and $B$.
    \end{enumerate}
\end{lemma}
\begin{proof}
    Set $D := \emptyset$ and $X' := V(G)$, and repeat the following: 
    \begin{itemize}        
        \item If there exists a partition $X' = R \cup A \cup B$ with $|R| \le \alpha |X'| / 8$ and $|A|, |B| \ge \alpha n/64$, such that there is no edge between $A$ and $B$, set $D := D \cup R$ and take $X'$ to be smaller of the two sets $A$ and $B$. (Note that we indeed bound $R$ in terms of $|X'|$, and $|A|$ and $|B|$ in terms of $n$.) Proceed to the next iteration.
        
        \item Otherwise, terminate the procedure.
    \end{itemize}
    The set $X'$ shrinks by a factor of $2$ or more in each iteration, thus the set $D$ is of size at most $|D| \le \alpha n / 4$. Once $|X'| < \alpha n / 32$, the procedure terminates as there is no partition $X' = R \cup A \cup B$ with $|A|, |B| \ge \alpha n / 64$. In particular, this implies that the procedure eventually terminates. By the definition of the procedure and the lower bound on $t$, we also have $|X'| \ge \alpha n / 64 \ge \alpha n / t$. As we shall see shortly, $X'$ is, in fact, even somewhat larger than this.
    
     Set $\gamma = \alpha n / |X'|$ and let $G' := G[X']$, for the final set $X'$. The crucial property of the procedure is that $N_G(v) \subseteq X' \cup D$ for every $v \in X'$. Using that $G$ is an $(\alpha, t)$-expander and $\gamma |X'| = \alpha n$, for any $S \subseteq X'$ of size $\gamma |X'| / (2t) \le |S| \le \gamma |X'| / t$ we have
    $$
        |N_{G'}(S)| = |N_G(S) \cap X'| \ge |N_G(S)| - |D| \ge t|S|/2.
    $$
    It is important to note here that $|X'| \ge \gamma |X'| / t$, thus such a set $S \subseteq X'$ indeed exists. By taking $S \subseteq X'$ to be a set of size $\gamma |X'| / t$, we conclude $|X'| > |N_G'(S)| \ge \gamma |X'| / 2$. This implies $\gamma / 2 < 1$.
    
    By Lemma \ref{lemma:one2all} applied on $G'$ with $\gamma/2$ (as $\alpha$) and $t/2$, which we can do as we have just established that $\gamma/2 < 1$, there exists a subset $R' \subseteq X'$ of size $|R'| \le \gamma |X'| / t$ such that $G'[X' \setminus R']$ is a $(\gamma/8, t/4)$-expander. From $t \ge 16 \cdot 64 / \alpha$ and the lower bound $|X'| \ge \alpha n / 64$, we have $|R'| \le \alpha |X'| / 16$. Set $X = X' \setminus R'$ and $\beta = \gamma / 8$. As $G[X]$ is a $(\beta, t/4)$-expander we necessarily have $\beta \le 1$, that is, $|X| \ge \alpha n / 8$.

    It remains to verify that the property \ref{prop:edge_across} holds. Consider some partition $X = R \cup A \cup B$ with $|R| \le \alpha |X| / 16$ and $|A|, |B| \ge \beta |X| / 4 > \gamma |X'|/64 = \alpha n / 64$. Then $|R \cup R'| \le \alpha |X'| / 8$, thus there is an edge between $A$ and $B$ by the fact that the procedure has terminated with $X'$.
\end{proof}

\section{Minors in expanders} \label{sec:minors}

Proofs of Theorem \ref{thm:complete_minors} and Theorem \ref{thm:expanding_minors} follow the overall strategy of Alon, Seymour, and Thomas \cite{alon90separator} and Plotkin, Rao, and Smith \cite{plotkin94shallow}. More closely, Theorem \ref{thm:expanding_minors} follows the proof of \cite[Theorem 8.1]{krivelevich19howtofind}. The proof of Theorem \ref{thm:complete_minors} is based on some further ideas that can be traced to the previous work of authors \cite{krivelevich09minors}. 

\subsection{Minor-universality}

To illustrate the main ideas, we start with the proof of Theorem \ref{thm:expanding_minors}.

\begin{proof}[Proof of Theorem \ref{thm:expanding_minors}]
    We first argue that it suffices to show that $G$ contains, as a minor, every graph $H$ with at most $3m$ vertices and at most $3m$ edges, and with maximum degree at most $3$. Consider a graph $H$ with at most $m$ vertices and at most $m$ edges, and suppose $V(H) = \{1, \ldots, v(H) \}.$ Form the graph $H'$ on the vertex set $V' = \{ (v, i) \colon v \in V(H), i \in \{1, \ldots, \deg_H(v)\} \} \cup V_0$, where the size of $V_0$ equals the number of isolated vertices in $H$. Add an edge between $(v,i)$ and $(w,j)$ if and only if $w$ is the $i$-th smallest neighbour of $v$, and $v$ is the $j$-th smallest neighbour of $v$. Furthermore, add an edge between $(v, i)$ and $(v, i+1)$ for each $v \in V(H)$ and $i \in \{1, \ldots, \deg_H(v) - 1\}$. Then $H$ is clearly a minor of $H'$. The graph $H'$ has maximum degree at most $3$, at most $3m$ vertices and at most $3m$ edges. As the relation of `being a minor' is transitive, it suffices to show that $G$ contains $H'$ as a minor.
    
    Let $X \subseteq V(G)$ be a subset given by Lemma \ref{lemma:robust_partition}. We now pass on to $\Gamma = G[X]$. Recall that $\Gamma$ has $N \ge \alpha n / 8$ vertices, and for $\beta = \alpha n / (8N)$ and $d = t/4$, the following holds:
    \begin{enumerate}[(i)]
        \item $\Gamma$ is a $(\beta, d)$-expander, and
        \item \label{prop:Gammaii} for every partition $V(\Gamma) = R \cup A \cup B$ with $|R| \le \alpha N/16$ and $|A|, |B| \ge \beta N / 4$, there is an edge in $\Gamma$ between $A$ and $B$.
    \end{enumerate}
    For later reference, note that 
    \begin{equation} \label{eq:alpha_beta}
        \beta \ge \alpha / 8 \quad \text{ and } \quad \alpha N / 32 \le \beta N / 4.
    \end{equation}

    Consider some graph $H$ with maximum degree at most $3$, at most $3m$ vertices, and at most $3m$ edges, where $m$ is the largest integer such that
    $$
        3m \cdot \frac{24 \log n}{\log d} < \beta \alpha N / 64 = \alpha^2 n / 512.
    $$
    The choice of $m$ implies $m = \Theta_\alpha(n \log t / \log n)$. Let $\mathcal{P}$ denote the set of all ordered partitions of the form $V(\Gamma) = D \cup \left( \bigcup_{h \in V'} W_h \right) \cup U$, where $V' \subseteq V(H)$ is some subset (it can be a different subset for different partitions), with the following properties:
    \begin{enumerate}[(P1)]
        \item \label{prop:W_v} for every $h \in V'$: $|W_h| \le 24 \log n / (\beta \log d)$ and $\Gamma[W_h]$ is connected, 
        \item \label{prop:W_e} if $\{h, h'\} \in E(H)$ for some $h, h' \in V'$, then there is an edge in $\Gamma$ between $W_h$ and $W_{h'}$, and
        \item \label{prop:D} $|D| \le \alpha N / (32 d) $ and $|N_\Gamma(D) \cap U| < d|D| / 2$.
    \end{enumerate}
    Properties \ref{prop:W_v} and \ref{prop:W_e} imply that $H[V']$ is a minor of $\Gamma$. 
    We frequently use the fact that, for any partition in $\mathcal{P}$, the set $W = \bigcup_{h \in V'} W_h$ is of size
    \begin{equation} \label{eq:W}
        |W| = \left| \bigcup_{h \in V'} W_h \right| = \sum_{h \in V'} |W_h| \le |V(H)| \cdot 24 \log n / (\beta \log d) \le \alpha N / 64.
    \end{equation}
    
    By taking $D = \emptyset$, $V' = \emptyset$, and $U = V(\Gamma)$, we see that $\mathcal{P}$ is not empty. Consider a partition $V(\Gamma) = D \cup \left( \bigcup_{h \in V'} W_h \right) \cup U$ in $\mathcal{P}$ which maximises $|D|$, and in the case there are multiple such partitions take one which further maximises $|V'|$. We prove that then necessarily $V' = V(H)$, which implies that $H$ is a minor of $G$.

    The following claim is stronger than necessary for the proof of Theorem \ref{thm:expanding_minors}. However, it is required in the proof of Theorem \ref{thm:complete_minors} thus we state it in this form.
    
    \begin{claim} \label{claim:v}
        For every $h \in V'$ we have $|N_\Gamma(W_h) \cap U| \ge d|W_v|/2$. 
    \end{claim}
    \begin{proof}
        Suppose, towards a contradiction, that this is not the case for some $v \in V'$. Then
        \begin{equation} \label{eq:D_Wv}
            |N_\Gamma(D \cup W_v) \cap U| < d|D \cup W_v|/2.
        \end{equation}
        If $|D \cup W_v| \le \alpha N / (32 d)$, by setting $D := D \cup W_v$ and $V' := V' \setminus \{v\}$, we obtain a partition in $\mathcal{P}$ with a larger set $D$ -- which is a contradiction. Therefore, we can assume $|D \cup W_v| > \alpha N / (32 d)$. We also have the following upper bound,
        $$
            |D \cup W_v| \le \frac{\alpha N}{32 d} + \frac{24 \log n}{\beta \log d} \le \frac{\beta N}{4d} + 192 \alpha^{-1} \log n < \beta N / d.
        $$
        The first inequality follows from \ref{prop:D} and \ref{prop:W_v}, the second from \eqref{eq:alpha_beta}, and the third from the fact that $n$ is sufficiently large with respect to $\alpha$. As $\Gamma$ is a $(\beta, d)$-expander, we conclude 
        $$
            |N_\Gamma(D \cup W_v)| \ge d|D \cup W_v| > \alpha N / 32.
        $$
        From \eqref{eq:W} we get
        $$
            |N_\Gamma(D \cup W_v) \cap U| \ge |N_\Gamma(D \cup W_v)| - |W| > d|D \cup W_v|/2.
        $$
        This contradicts \eqref{eq:D_Wv}.
    \end{proof}

    \begin{claim} \label{claim:small_diameter}
        The subgraph $\Gamma[U]$ is a connected $(\beta/2, d/2)$-expander. In particular, its diameter is at most $12 \log n / (\beta \log d)$.
    \end{claim}
    \begin{proof}
        Suppose first that $\Gamma[U]$ is not a $(\beta/2, d/2)$-expander. Then there there exists a subset $X \subseteq U$, $|X| \le \beta |U| / d$ such that $|N_\Gamma(X) \cap U| < d|X|/2$. Then
        \begin{equation} \label{eq:X}
            |N_\Gamma(D \cup X) \cap U| < d|D \cup X|/2. 
        \end{equation}
        Depending on the size of $D \cup X$, we obtain a contradiction as follows:
        \begin{itemize}
            \item $|D \cup X| \le \alpha N / (32d)$: Setting $D := D \cup X$ and $U := U \setminus X$ gives a partition in $\mathcal{P}$ with larger $D$. This contradicts the choice of the partition.
            \item $\alpha N / (32 d) \le |D \cup X| \le \beta N / d$: From the fact that $\Gamma$ is a $(\beta, d)$-expander and the upper bound \eqref{eq:W}, we conclude 
            $$
                |N_\Gamma(D \cup X) \cap U| = |N_\Gamma(D \cup X)| - |W| \ge d|D \cup X|/2.
            $$
            This contradicts \eqref{eq:X}.
            \item $|D \cup X| > \beta N / d$: From the upper bound $|X| \le \beta |U| / d$ and $|D| \le \beta N / (4d)$ (follows from \eqref{eq:alpha_beta}), we have $|D \cup X| \le \frac{5}{4} \beta N / d$. Let $D' \subseteq D \cup X$ be an arbitrary subset of size $|D'| = \beta N / d$. Then 
            $$
                |(D \cup X) \setminus D'| \le \beta N / (4d), 
            $$
            and so
            \begin{align*}
                |N_\Gamma(D \cup X) \cap U| &= |N_\Gamma(D')| - |(D \cup X) \setminus D'| - |W|  \\
                &\ge \beta N - \beta N / (4d) - \alpha N / 64 \ge 3 \beta N  / 4 \ge d|D \cup X|/2.
            \end{align*}
            Again, this contradicts \eqref{eq:X}.
        \end{itemize}
        
        We now show that $\Gamma[U]$ is connected. Since $\Gamma[U]$ is a $(\beta/2, d/2)$-expander, by Lemma \ref{lemma:ball} we conclude that the smallest connected component in $\Gamma[U]$ is of size at least $\beta |U| / 2 > \beta N / 4$. Suppose, towards a contradiction, that $\Gamma[U]$ is not connected. Let $A \subseteq U$ be the set of vertices of one component, and set $B = U \setminus A$. Then there is no edge in $\Gamma$ between $A$ and $B$. Consider the partition $V(\Gamma) = R \cup A \cup B$, where $R = D \cup W$. By the previous discussion, we have $|A|, |B| \ge \beta N / 4$. From \ref{prop:D} and \eqref{eq:W} we have $|R| \le \alpha N / 32$, and by the property \ref{prop:Gammaii} of $\Gamma$ we have that there is an edge between $A$ and $B$ --- thus a contradiction.

        To conclude, $\Gamma[U]$ is a connected $(\beta/2, d/2)$-expander. By Lemma \ref{lemma:diameter}, the diameter of $\Gamma[U]$ is most $12 \log |U| / (\beta \log d)$.
    \end{proof}

    Having these claims at hand, we proceed with proving $V' = V(H)$. Suppose, towards a contradiction, that $V' \neq V(H)$ and consider an arbitrary $h \in V(H) \setminus V'$. If $N_H(h) \cap V' = \emptyset$ then pick an arbitrary $v \in U$, and set $U := U \setminus \{v\}$, $W_h := \{v\}$ and $V' := V' \cup \{h\}$. This gives a partition in $\mathcal{P}$ with the same $D$ and larger $V'$, which is a contradiction. Therefore, we can assume $|N_H(h) \cap V'| \in \{1, 2,3\}$. Without loss of generality, assume $N_G(h) \cap V' = \{h_1, h_2, h_3\}$. Arbitrarily choose vertices $v_i \in N_\Gamma(W_{h_i}) \cap U$ for $i \in \{1,2,3\}$. This is possible due to Claim \ref{claim:v}. Let $L_i \subseteq U$ denote the vertices on a shortest path from $v_i$ to $v_{i+1}$ in $\Gamma[U]$, for $i \in \{1,2\}$. By Claim \ref{claim:small_diameter} we have $|L_i| \le 12 \log n / (\beta \log d)$.  Therefore, the set $W_h := L_1 \cup L_2$ is of size at most $24 \log n / (\beta \log d)$, $\Gamma[W_h]$ is connected, and there is an edge between $W_{h_i}$ and $W_h$ for every $i \in \{1,2,3\}$. Setting $U := U \setminus W_h$ and $V' := V' \cup \{h\}$, again, gives a partition in $\mathcal{P}$ with larger $V'$, which is a contradiction. Therefore, we conclude $V' = V(H)$.
\end{proof}

\subsection{Large complete minors}

Theorem \ref{thm:expanding_minors} implies a small-set expander contains a complete minor of order $\sqrt{n \log t / \log n}$, which matches \cite[Proposition 4.3]{krivelevich09minors}. Significantly improved bound in Theorem \ref{thm:complete_minors}, replacing $\log t$ with $t$, is based on a more efficient way of connecting a large number of large sets. This is captured by the following lemma, inspired by a similar statement for edge-expanders \cite[Lemma 3.1]{krivelevich21sparsecut}. 

\begin{lemma} \label{lemma:efficient_cover}
For every $0 < \alpha < 1$ there exists $s_0 > 1$ such that the following holds. Let $G$ be a connected $(\alpha, t)$-expander with $n$ vertices, for some $t \ge 2$. Let $U_1, \ldots, U_q \subseteq V(G)$ be subsets of size $|U_i| \ge s$, for some $1 \le q < n$ and $s_0 \le s \le n / \log n$ such that $qs \ge 2n$. Then there exists a set $T \subseteq V(G)$ of size at most
$$
    |T| \le \frac{25}{\alpha^2} \; \cdot \frac{n}{s} \log\left(\frac{qs}{n}\right) \cdot \frac{\log n}{\log t}, 
$$
such that $G[T]$ is connected and $T \cap U_i \neq \emptyset$ for every $i \in [q]$. 
\end{lemma}
\begin{proof}
    Let $P \subseteq V(G)$ be a subset obtained by taking each vertex in $G$, independently, with probability $p = 4 \log(q s / n) / (\alpha s)$. As $s$ is sufficiently large, we have $p < 1$. Let $X_i = \textrm{dist}_G(U_i, P)$. In particular, $X_i \ge z$, for some $z \ge 1$, is equivalent to $B(U_i, z-1) \cap P = \emptyset$. Note that $\Pr[X_i = z] = \Pr[X_i \ge z] - \Pr[X_i \ge z+1]$ and, as $G$ is connected, $\Pr[X_i \ge n] = 0$. Therefore,
    $$
        \mathbb{E}[X_i] = \sum_{z = 1}^n z \Pr[X_i = z] = \sum_{z = 1}^{n-1} z(\Pr[X_i \ge z] - \Pr[X_i \ge z+1]) = \sum_{z = 1}^{n-1} \Pr[X_i \ge z].
    $$
    As each vertex is included in $P$ independently, we have
    $$
        \Pr[X_i \ge z] = (1-p)^{|B_G(U_i, z-1)|} \le e^{-p|B_G(U_i, z-1)|} = \left( \frac{n}{q s} \right)^{4 \alpha^{-1} |B_G(U_i, z-1)| / s} .
    $$
    From $|B_G(U_i, z - 1)| \ge \min\{\alpha n, s t^{z-1}\}$, which follows from Lemma \ref{lemma:ball}, and $n/s \ge \log n$ we conclude
    \begin{equation} \label{eq:bound_short}
        \mathbb{E}[X_i] = \sum_{z = 1}^{n-1} \Pr[X_i \ge z] \le \sum_{z = 1}^{n-1} \left( \frac{n}{qs} \right)^{4 n / s} + \sum_{z = 1}^\infty \left( \frac{n}{qs} \right)^{4 t^{z-1}} \le  \frac{n}{qs}.
    \end{equation}

    Let $T' \subseteq V(G)$ be obtained by taking a shortest path from each $U_i$ to $P$. Then $|T' \setminus P| \le \sum_{i = 1}^q X_i$, thus
    $$
        \mathbb{E}[|T' \setminus P|] \le \sum_{i = 1}^q \mathbb{E}[X_i] \le \frac{n}{s}.
    $$
    As $\mathbb{E}[P] = pn$, by Markov's inequality we have $|P| \le 2np$ and $|T' \setminus P| \le 2n/s$ with positive probability. Therefore, there exists a choice of $P$ for which the two inequalities hold. Fix one such choice. Choose a vertex $v \in P$, and for each $w \in P \setminus \{v\}$ let $L_w$ denote the vertices on a shortest path from $v$ to $w$. By Lemma \ref{lemma:diameter}, the set $T := T' \cup \bigcup_{w \in P \setminus \{v\}} L_w$ is of size 
    $$
        |P| \frac{3 \log n}{\alpha \log t} + |T' \setminus P| \le \frac{6 np \log n}{\alpha \log t} + \frac{2 n}{s} \le  \frac{25}{\alpha^2} \cdot \; \frac{n}{s} \log\left(\frac{qs}{n} \right) \cdot \frac{\log n}{\log t}.
    $$
    By the construction, $G[T]$ is connected and $T$ intersects every set $U_i$.
\end{proof}

We are now ready to prove Theorem \ref{thm:complete_minors}.

\begin{proof}[Proof of Theorem \ref{thm:complete_minors}]
Let $X \subseteq V(G)$ be a subset given by Lemma \ref{lemma:robust_partition}. We pass on to $\Gamma = G[X]$. The graph $\Gamma$ has $N \ge \alpha n / 8$ vertices, and for $\beta = \alpha n / (8N)$ and $d = t/4$, the following holds:
\begin{enumerate}[(i)]
    \item $\Gamma$ is a $(\beta, d)$-expander, and
    \item for every partition $V(\Gamma) = R \cup A \cup B$ with $|R| \le \alpha N/16$ and $|A|, |B| \ge \beta N / 4$, there is an edge in $\Gamma$ between $A$ and $B$.
\end{enumerate}

Let $K = 25 / \alpha^2$ (the constant in the upper bound on the size of $T$ in Lemma \ref{lemma:efficient_cover}). Consider the family $\mathcal{P}$ of all ordered partitions $V(\Gamma) = D \cup \left( \bigcup_{i = 1}^k W_i \right) \cup U$, where $k \ge 0$ (this can be a different number for different partitions), such that the following holds:
\begin{enumerate}[(P1)]
    \item For every $i \in [k]$: $|W_i| = \sqrt{\frac{2 K N \log N}{d}} =: \ell$ and $\Gamma[W_i]$ is connected,
    \item for every two distinct $i, i' \in [k]$, there is an edge in $\Gamma$ between $W_i$ and $W_{i'}$, and
    \item $|D| \le \alpha N / (32 d)$ and $|N_\Gamma(D) \cap U| < d |D| / 2$.
\end{enumerate}
The first two properties imply that $K_k$ is a minor of $\Gamma$. 

By taking $D = \emptyset$, $k = 0$ and $U = V(\Gamma)$, we have that $\mathcal{P}$ is not empty. Consider a partition $V(\Gamma) = D \cup \left( \bigcup_{i = 1}^k W_i \right) \cup U$ in $\mathcal{P}$ which maximises $|D|$, tie-breaking by taking one which further maximises $k$. We prove that then necessarily 
$$
    k \ge \frac{\alpha N}{64 \ell} =: q.
$$
As $q = \Theta(\sqrt{n t / \log n})$, this establishes the theorem. Suppose, towards a contradiction, that this is note the case. That is, $k < q$. Then
\begin{equation*}
    |W| = \left| \bigcup_{i = 1}^k W_i \right| < \alpha N / 64,
\end{equation*}
from which we conclude $|U| \ge N/2$, with room to spare.

The property \ref{prop:D} in the proof of Theorem \ref{thm:expanding_minors} is identical to the one used here, and the only property of $W$ used in the proof of Claim \ref{claim:v} and Claim \ref{claim:small_diameter} is the upper bound on $|W|$, which is identical to the one used here. Therefore, same as in the proof of Theorem \ref{thm:expanding_minors}, the following holds:
\begin{itemize}
    \item For every $i \in [k]$ we have $|N_\Gamma(W_i) \cap U| \ge d|W_i|/2$, and
    \item $\Gamma[U]$ is a connected $(\beta/2, d/2)$-expander.
\end{itemize}
Now we can finish the proof using Lemma \ref{lemma:efficient_cover}. For each $i \in [k]$, set $U_i = N_\Gamma(W_i) \cap U$. Then $|U_i| \ge  d \ell / 2 =: s$. For $k < i \le q$, take $U_i \subseteq U$ to be an arbitrary set of size $s$. Apply Lemma \ref{lemma:efficient_cover} with sets $U_1, \ldots, U_q$, which we indeed can do as $qs > 2|U|$ and $|U| / s > \log |U|$, where the former follows from $d \ge t_0(\alpha)/4$ and the latter follows from $d \le \sqrt{n}$. We obtain a subset $T \subseteq U$ of size
$$
    |T| \le K \frac{|U|}{s} \log\left( \frac{qs}{|U|} \right) \cdot \frac{\log |U|}{\log d} \le \ell,
$$
such that $\Gamma[T]$ is connected and $T \cap U_i \neq \emptyset$ for every $i \in [k]$. Since $\Gamma[U]$ is connected, we can add more vertices from $U$ to $T$ such that $T$ remains connected and is of size exactly $\ell$. By setting $W_{k+1} := T$ and $U := U \setminus T$, we obtain a partition in $\mathcal{P}$ with the same size of the set $D$ and larger $k$, which is a contradiction. Therefore, $K_q$ is a minor of $\Gamma$.
\end{proof}

\section{Upper bound on minor universality} \label{sec:upper}

\begin{proof}[Proof of Theorem \ref{thm:upper_bound}]  
If $G$ is not connected, then add at most $n-1$ edges such that it becomes connected. The average degree of the resulting graph $G$ is at most $d+2$.

We can assume $d + 2 \le n^{1/6}$, as otherwise $m > n$ and assertion of the theorem trivially holds. We will argue that the number of minors of graphs with $m/2$ vertices (we tacitly assume $m$ is even) and $m$ edges in $G$ is less than the number of non-isomorphic graphs with so many vertices and edges. Therefore, at least one such graph is not a minor of $G$.

Let us estimate from above the number of minors of $G$ with $m/2$ vertices and $m$ edges. We first choose $m/2$ disjoint connected sets $B_1,\ldots,B_{m/2}$ in $G$. Since $G$ is connected we may assume that $\bigcup_{i=1}^{m/2}B_i=V(G)$. Given $B_1,\ldots,B_{m/2}$ as above, we can choose a spanning tree $T_i$ in each of $B_i$'s, and add to their union $m/2-1$ edges of $G$ to create a spanning tree $T$ of $G$ (this is possible due to $G$ being connected). Reversing the process, we can estimate the number of ways to choose $B_1,\ldots,B_{m/2}$ as follows:
\begin{enumerate}
\item Choose a spanning tree $T$ of $G$. This can be done in at most $\prod_{v\in V} d_G(v)\le (d+2)^n$ ways (see, e.g., \cite{kostochka95count});
\item Choose $m/2-1$ edges to delete from $T$. This can be done in $\binom{n-1}{m/2-1}$ ways.
\end{enumerate}
Having fixed the blocks $B_1,\ldots, B_{m/2}$, we can choose $m$ edges between them in at most 
$$
    \binom{|E(G)|}{m} \le \binom{(d+2)n/2}{m}
$$ 
ways. Altogether, the number of minors of $G$ with $m/2$ vertices and $m$ edges is at most
\begin{equation} \label{eq:minor_ub_1}
(d+2)^n\binom{n-1}{m/2-1}\binom{(d+2)n/2}{m} < (d+2)^{n}\left(\frac{2en}{m}\right)^{m/2} \left(\frac{2(d+2)n}{m}\right)^{m}\,.    
\end{equation}

Next, we derive a lower bound on the number of non-isomorphic graphs with $m/2$ vertices and $m$ edges. Since every graph with $m/2$ labelled vertices is isomorphic to at most $(m/2)! < m^{m/2}$ graphs on the same vertex set, we derive that the number of non-isomorphic graphs with $m/2$ vertices and $m$ edges is at least
\begin{equation} \label{eq:minod_ub_2}
\frac{\binom{\binom{m/2}{2}}{m}}{(m/2)!}> \frac{\left(\frac{m}{16}\right)^{m}}{m^{m/2}}=\left(\frac{m}{256}\right)^{m/2}\,.
\end{equation}

It remains to compare \eqref{eq:minor_ub_1} and \eqref{eq:minod_ub_2}. Recalling the value of $m$, the assumption $d + 2 \le n^{1/6}$, and that $n$ is sufficiently large, we derive:
\begin{eqnarray*}
\left(\frac{m}{256}\right)^{m/2} \left(\frac{m}{2en}\right)^{m/2} \left(\frac{m}{2(d+2)n}\right)^{m} &=& \left(\frac{m^4}{2^{11} e (d+2)^2 n^3 }\right)^{m/2}  > \left(\frac{n^4 \ln^4(d+2)}{2^{11} e (d+2)^2 n^3 \ln^4n}\right)^{m/2}\\
&>& (n^{3/5})^{m/2} > e^{n\ln (d+2)} = (d+2)^{n}\,.
\end{eqnarray*}
We conclude there exists a graph $H$ with $m/2$ vertices and $m$ edges which is not a minor of $G$. In fact, our argument shows that most such graphs are not minors of $G$.
\end{proof}

\bibliographystyle{abbrv}
\bibliography{minor_universality_final}

\appendix

\section{Expansion of $(n,d,\lambda)$-graphs}

Given a graph $G$ and two subsets $X, Y \subseteq V(G)$, we denote with $e(X, Y)$ the number of pairs of vertices $(x,y) \in X \times Y$ such that $\{x, y\} \in E(G)$. We make use of the Expander Mixing Lemma \cite{alon88mixing}.

\begin{lemma}[Expander Mixing Lemma]
    Let $G$ be an $(n, d, \lambda)$-graph. Then for every two sets $X, Y \subseteq V(G)$, we have
    $$
        e(X, Y) \le \frac{d}{n} |X||Y| + \lambda \sqrt{|X||Y|}.
    $$
\end{lemma}

The following lemma shows that $(n, d, \lambda)$-graphs, for $\lambda < d/4$, are small-set expanders.
\begin{lemma} \label{lemma:ndl_expander}
    Let $G$ be an $(n, d, \lambda)$-graph with $\lambda < d/4$. Then $G$ is a $(1/4, d^2/(4\lambda)^2)$-expander.
\end{lemma}
\begin{proof}
    Set $t := d^2 / (4\lambda)^2$. Consider some $X \subseteq V(G)$ of size $|X| \le n / (4t)$. By the Expander Mixing Lemma, the upper bound on the size of $X$, and the upper bound on $\lambda$, we have
    $$
        e(X, X) \le \frac{d}{n} \cdot \frac{n}{4t} |X|  + \lambda |X| = \frac{4 \lambda^2}{d} |X| + \lambda |X| < d|X|/2.
    $$
    Let $Y = N(X)$. The previous upper bound on $e(X,X)$ implies $e(X, Y) > d|X|/2$. Suppose, towards a contradiction, that $|Y| < t|X|$. Applying the Expander Mixing Lemma once again, we get
    $$
        e(X, Y) \le \frac{d}{n} |X|^2 t + \lambda |X| \sqrt{t} \le \frac{d}{n} \cdot \frac{n}{4t} |X|t + d |X| /4 < d|X|/2.
    $$
    This contradicts the lower bound $e(X, Y) \ge d|X|/2$.
\end{proof}

\end{document}